\documentclass[10pt]{amsart}
\usepackage{amssymb, amsthm, amsmath}
\usepackage[all]{xy}
\usepackage{graphicx}
\usepackage{psfrag}

\newtheorem{prop}{Proposition}
\newtheorem{thm}{Theorem}
\newtheorem{cor}{Corollary}
\newtheorem{lemma}{Lemma}
\theoremstyle{definition}
\newtheorem{defn}{Definition}
\newtheorem{example}{Example}

\newcommand\A{{\mathbb A}}

\newcommand\C{{\mathbb C}}

\newcommand\N{{\mathbb N}}
\newcommand{\cN}{{\mathcal N}}

\newcommand{\ti}{\vartheta}
\newcommand{\Ti}{\Theta}

\newcommand\cC{{\mathcal C}}

\newcommand\cW{{\mathcal W}}

\newcommand\Z{{\mathbb Z}}
\newcommand\cP{{\mathcal P}}
\newcommand\cQ{{\mathcal Q}}

\newcommand\cR{{\mathcal R}}
\newcommand\Eta{H}

\newcommand\DS{{\mathfrak D}}

\newcommand\al{\alpha}
\newcommand{\be}{\beta}
\newcommand\la{\lambda}

\newcommand\om{{\omega}}
\newcommand\Om{{\Omega}}

\newcommand\ssm{\smallsetminus}
\newcommand\gequ{\geq}

\newcommand\noin{\noindent}

\newcommand\eqto{\stackrel{\lower1.5pt\hbox{$\scriptstyle\sim\,$}}\to}
\newcommand\ov{\overline}

\newcommand\rsa{\rightsquigarrow}
\newcommand\wh{\widehat}
\newcommand\wt{\widetilde}
\newcommand\dis{\displaystyle}

 \DeclareMathOperator{\SO}{SO}

 \DeclareMathOperator{\OG}{OG}
 
\DeclareMathOperator{\HH}{\mathrm{H}} 
\DeclareMathOperator{\type}{\mathrm{type}}

\newcommand{\ignore}[1]{}
\newcommand{\pic}[2]{\includegraphics[scale=#1]{#2}}

\psfrag{k}{$k$}
\psfrag{(r,c)}{$[r,c]$}
\psfrag{(r',c')}{$[r',c']$}
\psfrag{la1}{$\la^1$}
\psfrag{la2}{$\la^2$}

\begin{document}

\title[A tableau formula for eta polynomials]
{A tableau formula for eta polynomials}

\date{November 18, 2013}

\author{Harry~Tamvakis} \address{University of Maryland, Department of
Mathematics, 1301 Mathematics Building, College Park, MD 20742, USA}
\email{harryt@math.umd.edu}


\thanks{The author was supported in part by NSF Grants DMS-0901341
and DMS-1303352}

\begin{abstract}
We use the Pieri and Giambelli formulas of \cite{BKT1, BKT3} and the
calculus of raising operators developed in \cite{BKT2, T1} to prove a
tableau formula for the eta polynomials of \cite{BKT3} and the Stanley
symmetric functions which correspond to Grassmannian elements of the
Weyl group $\wt{W}_n$ of type $\text{D}_n$. We define the {\em skew
  elements} of $\wt{W}_n$ and exhibit a bijection between the set of
reduced words for any skew $w\in \wt{W}_n$ and a set of certain
standard typed tableaux on a skew shape $\la/\mu$ associated to $w$.
\end{abstract}

\maketitle

\section{Introduction}

Let $k$ and $n$ be positive integers with $k\leq n$ and
$\OG=\OG(n+1-k,2n+2)$ be the even orthogonal Grassmannian which
parametrizes isotropic subspaces of dimension $n+1-k$ in a complex
vector space of dimension $2n+2$, equipped with a nondegenerate
symmetric bilinear form.  Following \cite{BKT1}, the Schubert classes
$\sigma_\lambda$ in $\HH^*(\OG,\Z)$ are indexed by {\em typed
  $k$-strict partitions} $\la$. Recall that a $k$-strict partition is
an integer partition $\la=(\la_1,\ldots,\la_\ell)$ such that all parts
$\la_i$ greater than $k$ are distinct. A typed $k$-strict partition is
a pair consisting of a $k$-strict partition $\la$ together with an
integer $\type(\la)\in \{0,1,2\}$ which is positive if and only if
$\la_i=k$ for some index $i$.

In \cite{BKT3}, we discovered a remarkable connection between the
cohomology of even (type D) and odd (type B) orthogonal Grassmannians,
which allows a uniform approach to the Schubert calculus on both
varieties. We used this together with our earlier work \cite{BKT2} to
obtain a {\em Giambelli formula} for the Schubert class $\sigma_\la$,
which expresses it as a polynomial in certain special Schubert
classes. This polynomial, which is called an {\em eta polynomial}, and
denoted $H_\la$, is defined using Young's raising operators \cite{Li,
Y}. These Giambelli polynomials naturally live in the stable
cohomology ring of $\OG$, and multiply like the Schubert classes on
$\OG(n+1-k,2n+2)$ when $n$ is sufficiently large. 

The present paper is concerned with the principal specialization
$H_\la(x\,;y)$ of $H_\la$ in the ring of type D Billey-Haiman Schubert
polynomials \cite{BH}. In \cite{BKT3}, we proved that the polynomial
$H_\la(x\,;y)$ may be identified with the Billey-Haiman Schubert
polynomial $\DS_{w_\la}(x\,;y)$ indexed by the corresponding
$k$-Grassmannian element $w_\la$ in the Weyl group $\wt{W}_{n+1}$ for
the root system of type $\text{D}_{n+1}$. Note that
$\DS_{w_\la}(x\,;y)$ is really a formal power series, and that its
equality with $H_\la(x\,;y)$ holds only modulo the relations among the
Schur $P$-functions which enter into the definition of $H_\la(x\,;y)$
(see \S \ref{wsec} for more details). A significant application of
these eta polynomials had appeared earlier in \cite[\S 6]{T2}, where
they were used in splitting formulas for the general type D Schubert
polynomials $\DS_w(x\,;y)$. This result was applied in \cite{T2} to
prove combinatorially explicit Chern class formulas for degeneracy
loci of vector bundles in the sense of Fulton \cite{Fu}. The reader
may consult \cite{T3} for an exposition of this work, which covers
all the classical Lie groups.

The first goal of this article is to combine the above algebraic and
combinatorial theory with the raising operator approach to tableau
formulas developed in \cite{T1}. The constructions here are analogous
to the type C case studied in op.\ cit., but we take the opportunity
to clarify and simplify some of the earlier proofs, when treating the
corresponding results (especially Theorem \ref{mirrorW} and Theorem
\ref{abprop}).

Our main theorem (Theorem \ref{Htableauxthm}) is a formula for the eta
polynomial $H_\la(x\,;y)$, which writes it as a sum of monomials
$(xy)^U$ over certain fillings $U$ of the Young diagram of $\la$
called {\em typed $k'$-bitableaux}. The analysis turns out to be
rather subtle, since a direct synthesis of the results of \cite{BKT3}
and \cite{T1} does not lead to {\em positive} formulas (compare with
\cite[Prop.\ 5.4]{BKT3}). We obtain the key reduction formula
for $H_\la$ (Theorem \ref{reductthmH}) by combining the reduction
formulas for $\wh{H}_\la$ and $\wt{H}_\la$, which are polynomial
constituents of $H_\la$. It is not a priori clear that such an
approach will work; the resulting story is the main innovation of this
paper.

We call an element $w$ of $\wt{W}_{n+1}$ {\em skew} if there exists a
$w'\in\wt{W}_{n+1}$ and a $k$-Grassmannian element $w_\la$ such that
$ww'=w_\la$ and $\ell(w)+\ell(w')=|\la|$.  The skew elements of the
symmetric group are precisely the $321$-avoiding permutations or fully
commutative elements, which were introduced in \cite{BJS} and explored
further in \cite{Ste1, Ste2}.  In the other classical Lie types,
although every fully commutative element is skew, the converse is
false. The skew elements of the the hyperoctahedral group were studied
in \cite{T1}, and we extend this theory here to $\wt{W}_{n+1}$.

Let $\la$ and $\mu$ be typed $k$-strict partitions such that the diagram
of $\mu$ is contained in the diagram of $\la$. Our approach leads
naturally to the definition of certain symmetric functions
$E_{\la/\mu}(x)$, given as a sum of monomials $x^T$ corresponding to
{\em typed $k'$-tableaux} $T$ on the skew shape $\la/\mu$.  We find
that the function $E_{\la/\mu}(x)$, when non-zero, is equal to a type
D Stanley symmetric function $E_w(x)$ indexed by a skew element $w$ in
$\wt{W}_{n+1}$. As in \cite{BJS} and \cite{T1}, there is an explicit
bijection between the standard typed $k'$-tableaux on the skew shape
$\la/\mu$ and the reduced words for a corresponding skew element $w\in
\wt{W}_{n+1}$.

This paper is organized as follows. Section \ref{typeC:general}
reviews the Giambelli and Pieri formulas which hold in the Chern
subring $\Om^{(k)}$ of the stable cohomology ring of $\OG(n+1-k,2n+2)$
as $n\to\infty$. We also establish the {\em mirror identity} in
$\Om^{(k)}$, a key technical tool which provides a bridge between the
Pieri rule and our tableau formulas. Section \ref{red&tab} introduces
the eta polynomials and proves various reduction formulas which are
then combined to obtain our main tableau formula for
$\Eta_\la(x\,;y)$. Finally, in Section \ref{sss} we relate this theory
to the type D Schubert polynomials and Stanley symmetric functions of
\cite{BH, L}, and study the skew elements of $\wt{W}_{n+1}$.

The author is grateful to the referee for a careful reading of the
paper and for suggestions which helped to improve the exposition.

\section{Preliminary results}
\label{typeC:general}

\subsection{Raising operators} An {\em integer sequence} is a sequence
of integers $\al=\{\al_i\}_{i\geq 1}$, only finitely many of which are
non-zero. The {\em length} of $\al$, denoted $\ell(\al)$, is the
largest integer $\ell\geq 0$ such that $\al_\ell\neq 0$. We will
identify an integer sequence of length $\ell$ with the vector
consisting of its first $\ell$ terms.  We set $|\al| = \sum \al_i$ and
let $\#\al$ equal the number of non-zero parts $\al_i$ of $\al$. The
inequality $\al\geq \be$ means that $\al_i \geq \be_i$ for each $i$.
We say that $\al$ is a {\em composition} if $\al_i\geq 0$ for all $i$
and a {\em partition} if $\al_i \geq \al_{i+1}\geq 0$ for all $i$. We
will represent a partition $\la$ by its Young diagram of boxes, which
has $\la_i$ boxes in row $i$ for each $i\geq 1$. The containment
relation between two Young diagrams is denoted by $\mu\subset\la$
instead of $\mu\leq\la$; in this case the set-theoretic difference
$\la\ssm\mu$ is called a skew diagram and is denoted by $\la/\mu$.

Given any integer sequence $\alpha=(\alpha_1,\alpha_2,\ldots)$ and
$i<j$, we define
\[
R_{ij}(\alpha) = (\alpha_1,\ldots,\alpha_i+1,\ldots,\alpha_j-1,
\ldots);
\] 
a raising operator $R$ is any monomial in these $R_{ij}$'s.  Given any
formal power series $\sum_{i \geq 0} c_i t^i$ in the variable $t$ and
an integer sequence $\al = (\al_1,\al_2,\dots,\al_\ell)$, we write
$c_\al = c_{\al_1} c_{\al_2} \cdots c_{\al_\ell}$ and set $R\,c_{\al}
= c_{R\al}$ for any raising operator $R$.  We will always work with
power series with constant term 1, so that $c_0=1$ and $c_i=0$ for
$i<0$.

\subsection{The Giambelli formula} 

Fix an integer $k > 0$. We consider an infinite family
$\om_1,\om_2,\ldots$ of commuting variables, with $\om_i$ of degree $i$ for
all $i$, and set $\om_0=1$, $\om_r=0$ for $r<0$, and $\om_\al= \prod_i
\om_{\al_i}$.  Let $I^{(k)}\subset \Z[\om_1,\om_2,\ldots]$ be the
ideal generated by the relations
\[
\frac{1-R_{12}}{1+R_{12}}\, \om_{(r,r)} = 
\om_r^2 + 2\sum_{i=1}^r(-1)^i \om_{r+i}\om_{r-i}= 0
\ \ \ \text{for} \  r > k.
\]
Define the graded ring $\Om^{(k)}=\Z[\om_1,\om_2,\ldots]/{I^{(k)}}$.

Let $\Delta^{\circ} = \{(i,j) \in \N \times \N \mid 1\leq i<j \}$ and
equip $\Delta^{\circ}$ with 
the partial order $\leq$ defined by $(i',j')\leq (i,j)$
if and only if $i'\leq i$ and $j'\leq j$.  A finite subset $D$ of
$\Delta^{\circ}$ is a {\em valid set of pairs} if it is an order
ideal, i.e., $(i,j)\in D$ implies $(i',j')\in D$ for all $(i',j')\in
\Delta^{\circ}$ with $(i',j') \leq (i,j)$. For any valid set of pairs
$D$, we define the raising operator
\[
R^D = \prod_{i<j}(1-R_{ij})\prod_{i<j\, :\, (i,j)\in D}(1+R_{ij})^{-1}.
\]

A partition $\la$ is {\em $k$-strict} if all its parts greater than
$k$ are distinct. The {\em $k$-length} of $\la$, denoted
$\ell_k(\la)$, is the cardinality of the set $\{i\ |\ \la_i > k\}$. We
say that $\la$ has {\em positive type} if $\la_i = k$ for some index
$i$.  The monomials $\om_\la$ for $\la$ a $k$-strict partition form a 
$\Z$-basis of $\Om^{(k)}$.

Given a $k$-strict partition $\la$, we define a valid set of pairs
$\cC(\la)$ by
\[
\cC(\la) = \{(i,j)\in\Delta^{\circ}\ |\ \la_i+\la_j \geq 2k+j-i \ \,
\text{and} \ \, j \leq \ell(\la)\}
\]
and set $R^\la := R^{\cC(\la)}$.  Furthermore, define $\Om_{\la}\in
\Om^{(k)}$ by the {\em Giambelli formula}
\begin{equation}
\label{giambelliCgen}
\Om_{\la} := R^\la\, \om_{\la}.
\end{equation}
Since
\[
\Om_\la = \om_\la + \sum_{\mu\succ\la} b_{\la\mu} \, \om_\mu
\]
with the sum over $k$-strict partitions $\mu$ which strictly dominate
$\la$, it follows that the $\Om_\la$ as $\la$ runs over $k$-strict partitions
form another $\Z$-basis of $\Om^{(k)}$.  More generally, given any valid
set of pairs $D$ and an integer sequence $\alpha$, we denote $R^D\,
\om_{\al}$ by $\Om^D_\al$.

\begin{example} In the ring $\Om^{(2)}$ we have
\begin{align*}
\Om_{322} 
&= \frac{1-R_{12}}{1+R_{12}}(1-R_{13})(1-R_{23})\, \om_{322} \\
&= (1-2R_{12}+2R_{12}^2 - \cdots)
(1-R_{13})(1-R_{23}) \, \om_{322} \\
&= (1-2R_{12}+2R_{12}^2 - 2 R_{12}^3)(1-R_{13}-R_{23}+R_{13}R_{23})\, 
\om_{322} \\
&= \om_{322}-\om_{421}-2\,\om_7+2\,\om_{61}-\om_{331}+\om_{43} \\
&= \om_3\,\om_2^2 -\om_4\,\om_2\,\om_1 -2\,\om_7 +2\,\om_6\,\om_1 
-\om_3^2\,\om_1+\om_4\,\om_3.
\end{align*}
\end{example}

\subsection{Cohomology of even orthogonal Grassmannians}
\label{cohevog}

Let the vector space $V=\C^{2n+2}$ be equipped with a nondegenerate
symmetric bilinear form. A subspace $\Sigma$ of $V$ is called {\em
isotropic} if the form vanishes when restricted to $\Sigma$. The
dimensions of such isotropic subspaces $\Sigma$ range from $0$ to $n+1$.

Choose $k$ with $0< k \leq n$ and let $\OG=\OG(n+1-k,2n+2)$ denote the
Grassmannian parametrizing isotropic subspaces of $V$ of dimension
$n+1-k$. Let $\cQ$ denote the universal quotient vector bundle of rank
$n+1+k$ over $\OG$, and $\HH^*(\OG,\Z)_1$ denote the subring of the
cohomology ring $\HH^*(\OG,\Z)$ which is generated by the Chern
classes of $\cQ$. There is a ring epimorphism $\psi: \Om^{(k)} \to
\HH^*(\OG,\Z)_1$ which maps the generators $\om_p$ to the Chern classes
$c_p(\cQ)$ for each $p\geq 1$; in particular we have $\psi(\om_p)=0$ if
$p > n+k$.

It follows from \cite[Theorem 1]{BKT3} that for any $k$-strict
partition $\la$, we have 
\[
\psi(2^{-\ell_k(\la)}\,\Om_\la) = [Y_\la]
\]
if the diagram of $\la$ fits inside a rectangle $\cR$ of size
$(n+1-k)\times (n+k)$, and $\psi(\Om_\la)=0$ otherwise. Here $Y_\la$
is a certain Zariski closed subset of pure codimension $|\la|$ in
$\OG$, which is a Schubert variety in $\OG$, if $\la$ does not have
positive type, and a union of two Schubert varieties in $\OG$, if
$\la$ has positive type. The cohomology classes $[Y_\la]$ for $\la$ a
$k$-strict partition contained in $\cR$ form a $\Z$-basis for
$\HH^*(\OG,\Z)_1$.  More details and the proofs of these facts are
provided in \cite{BKT3}.

\subsection{The Pieri rule} 
\label{pierirulegen}
We let $[r,c]$ denote the box in row $r$ and column $c$ of a Young
diagram. We say that the boxes $[r,c]$ and $[r',c']$ are {\em
$k'$-related} if 
$$\left|c-k-\frac{1}{2}\right|+r=\left|c'-k-\frac{1}{2}\right|+r'.$$
In the diagram of Figure \ref{kprshift}, the two grey boxes are
$k'$-related. The definition of $k'$-related boxes was introduced in
\cite[\S 3.2]{BKT1}, and is a type D analogue of the notion of
$k$-related boxes, which is used in the Lie types B and C. The
equivalence relation `$k'$-related' is the same as `$k$-related' when
$k$ is replaced by the half integer $k-1/2$. We will define
`$k$-related' precisely in \S \ref{Adef}.
\begin{figure}
\centering
\pic{0.65}{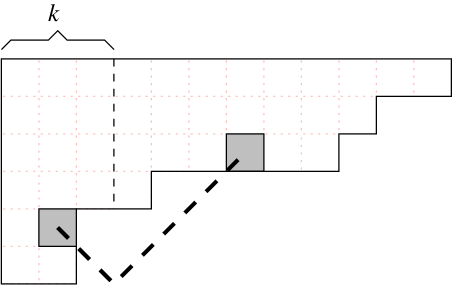} 
\caption{Two $k'$-related boxes in a $k$-strict Young diagram}
\label{kprshift}
\end{figure}

Given two partitions $\la$ and $\mu$ with $\la\subset\mu$, the skew
Young diagram $\mu/\la$ is called a {\em horizontal strip}
(respectively, {\em vertical strip}) if it does not contain two boxes
in the same column (respectively, row). For any two $k$-strict
partitions $\la$ and $\mu$, let $\al_i$ (respectively $\be_i$) denote
the number of boxes of $\la$ (respectively $\mu$) in column $i$, for
$1\leq i\leq k$.  We have a relation $\lambda \to \mu$ if $\mu$ can be
obtained by removing a vertical strip from the first $k$ columns of
$\lambda$ and adding a horizontal strip to the result, so that for
each $i$ with $1\leq i\leq k$,

\medskip
\noin
(1) if $\be_i=\al_i$, then the box $[\al_i,i]$ 
is $k'$-related to at most one box of $\mu \smallsetminus \lambda$; and

\medskip
\noin 
(2) if $\be_i < \al_i$, then the boxes
$[\be_i,i],\ldots,[\al_i,i]$ must each be $k'$-related to
exactly one box of $\mu \smallsetminus \lambda$, and these boxes of
$\mu \smallsetminus \lambda$ must all lie in the same row.

\medskip
If $\lambda \to \mu$, we let $\A'$ be the set of boxes of $\mu\ssm \la$
in columns $k+1$ and higher which are {\em not} mentioned in (1) or
(2). Define the connected components of $\A'$ by agreeing that two
boxes in $\A'$ are connected if they share at least a vertex. Then
define $N(\lambda,\mu)$ to be the number of connected components of
$\A'$, and set
\[
M(\la,\mu) = \ell_k(\la)-\ell_k(\mu)+\begin{cases}
N(\lambda,\mu)+1 & \text{if $\la$ has
positive type and $\mu$ does not}, \\
N(\la,\mu) & \text{otherwise}.
\end{cases}
\]

We deduce from \cite[Theorem 5]{BKT3} and the discussion in \S
\ref{cohevog} that the following Pieri rule holds: For any $k$-strict
partition $\lambda$ and integer $p\geq 0$,
\begin{equation}
\label{finaleq}
 \om_p \cdot \Om_\lambda = \sum_{{\lambda \to \mu} \atop {|\mu|=|\lambda|+p}} 
 2^{M(\lambda,\mu)} \, \Om_\mu \,.
\end{equation}
To compare with \cite[\S 1]{BKT3}, observe that the notion of
$K$-related boxes used in loc.\ cit.\ agrees with the notion of
$k'$-related boxes when the dimension $N$ of the ambient vector space
is even (so we are in Lie type D).

For any $d\geq 1$ define the raising operator $R^{\la}_d$ by
\[
R_d^{\la} = \prod_{1\leq i<j\leq d}(1-R_{ij})\,
\prod_{i<j\, :\, (i,j)\in\cC(\la)} (1+R_{ij})^{-1}.
\]
We compute that
\[
\om_p\cdot \Om_\la = \om_p\cdot R_{\ell}^{\la}\, \om_{\la} =
R_{\ell+1}^\la\cdot \prod_{i=1}^\ell(1-R_{i,\ell+1})^{-1} \, \om_{\la,p}
\]
\[
= R^\la_{\ell+1}\cdot\prod_{i=1}^\ell(1+R_{i,\ell+1} +
R_{i,\ell+1}^2 + \cdots)\,\om_{\la,p} = 
\sum_{\nu\in \cN(\la,p)} \Om^{\cC(\la)}_\nu\,,
\]
where $\cN=\cN(\la,p)$ is the set of all compositions $\nu\geq \la$
such that $|\nu| = |\la|+p$ and $\nu_j =0$ for $j > \ell+1$, where
$\ell$ denotes the length of $\la$. Equation (\ref{finaleq}) is
therefore equivalent to the identity
\begin{equation}
\label{toshow}
\sum_{\nu\in \cN(\la,p)} \Om^{\cC(\la)}_\nu = \sum_{{\lambda \to \mu} \atop
  {|\mu|=|\lambda|+p}} 2^{M(\lambda,\mu)} \, \Om^{\cC(\mu)}_\mu \,,
\end{equation}
which was proved in \cite{BKT3}.

\subsection{The mirror identity}
\label{mirrorsec}
Let $\la$ and $\nu$ be $k$-strict partitions such that $$\nu_1 >
\max(\la_1,\ell(\la)+2k-1)$$ and choose $p,m\geq 0$. Then the Pieri
rule (\ref{finaleq}) implies that the coefficient of $\Om_\nu$ in the
Pieri product $\om_p \cdot \Om_\la$ is equal to the coefficient of
$\Om_{(\nu_1+m,\nu_2,\nu_3,\ldots)}$ in the product $\om_{p+m} \cdot
\Om_\la$.  We apply this to make the following important definition.

\begin{defn}
\label{kprhorstr}
Let $\la$ and $\mu$ be $k$-strict partitions with $\mu\subset\la$,
and choose any $p \gequ \max(\la_1+1,\ell(\la)+2k-1)$. If
$|\la|=|\mu|+r$ and $\la\to(p+r,\mu)$, then we write
$\mu\rsa\la$ and say that $\la/\mu$
is a {\em $k'$-horizontal strip}.  We define $m(\la/\mu):=
M(\la,(p+r,\mu))$; in other words, the numbers $m(\la/\mu)$ are
the exponents that appear in the Pieri product
\begin{equation}
\label{pprod}
\om_p\cdot \Om_\la = \sum_{r,\mu}2^{m(\la/\mu)}\,\Om_{(p+r,\mu)}
\end{equation}
with the sum over integers $r\geq 0$ and $k$-strict partitions
$\mu\subset\la$ with $|\mu| = |\la|-r$. 
\end{defn}

Note that a $k'$-horizontal strip $\la/\mu$ is a pair of partitions
$\la$ and $\mu$ with $\mu\rsa\la$. As such it depends on $\la$ and
$\mu$ and not only on the difference $\la\ssm\mu$.

\begin{lemma}
\label{basiclemD}
Let $\dis \Psi = \prod_{j=2}^{\ell+1} \frac{1-R_{1j}}{1+R_{1j}}$, 
and suppose that we have an equation
\[
\sum_\nu a_\nu \om_\nu = \sum_\nu b_\nu \om_\nu
\]
in $\Om^{(k)}$, where the sums are over all $\nu = (\nu_1, \ldots,
\nu_\ell)$, while $a_\nu$ and $b_\nu$ are integers only finitely many
of which are non-zero.  Then we have
\[
\sum_\nu a_\nu \Psi \,\om_{(p,\nu)} = \sum_\nu b_\nu \Psi \, \om_{(p,\nu)}
\]
in the ring $\Om^{(k)}$, for any integer $p$.
\end{lemma}
\begin{proof}       
The proof is the same as \cite[Proposition 2]{T1}.
\end{proof}

Let $\la$ be any $k$-strict partition of length $\ell$. Consider 
the following version of (\ref{toshow}):
\begin{equation}
\label{showed}
\sum_{\al\geq 0} \Om^{\cC(\la)}_{\la+\al} = 
\sum_{\lambda \to \mu} 2^{M(\lambda,\mu)} \, \Om^{\cC(\mu)}_\mu 
\end{equation}
where the first sum is over all compositions $\al$ of length at most
$\ell+1$, and the second over $k$-strict partitions $\mu$ with
$\la\to\mu$. The next result is called the {\em mirror identity} of 
(\ref{showed}), and is an analogue of  \cite[Theorem 2]{T1}; the proof 
we give below simplifies the one found in loc.\ cit.\ 

\begin{thm}
\label{mirrorW}
For $\la$ any $k$-strict partition we have
\begin{equation}
\label{minuseq}
\sum_{\al\geq 0} 2^{\#\al} \, \Om^{\cC(\la)}_{\la-\al} 
= \sum_{\mu\rsa\la} 2^{m(\la/\mu)} \, \Om_\mu,
\end{equation}
where the first sum is over all compositions $\al$.
\end{thm}
\begin{proof}       
Choose $p \geq |\la|+2k$ and let $\ell=\ell(\la)$. Expanding the Giambelli
formula with respect to the first row gives
\begin{gather*}
\om_p\cdot \Om_\la = 
 R^{\cC(p,\la)}
\,\prod_{j=2}^{\ell+1} \, \frac{1+R_{1j}}{1-R_{1j}} \, \om_{(p,\la)} =
 R^{\cC(p,\la)}
\,\prod_{j=2}^{\ell+1} (1+2 R_{1j}+2 R_{1j}^2+\cdots)\,\om_{(p,\la)} \\
= \sum_{\al\geq 0} 2^{\#\al} \, \Om^{\cC(p,\la)}_{(p+|\al|,\la-\al)}.
\end{gather*}
Comparing this with (\ref{pprod}), we deduce that 
\[
\sum_{\al\geq 0} 2^{\#\al} \, \Om^{\cC(p,\la)}_{(p+|\al|,\la-\al)} 
= \sum_{\mu\rsa\la} 2^{m(\la/\mu)} \, \Om_{(p+|\la|-|\mu|,\mu)}.
\]

We claim that for every integer $r\geq 0$, 
\begin{equation}
\label{pfeq1}
\sum_{|\al|=r} 2^{\#\al} \, \Om^{\cC(p,\la)}_{(p+r,\la-\al)} 
= \sum_{{\mu\rsa\la}\atop{|\mu|=|\la|-r}} 
2^{m(\la/\mu)} \, \Om_{(p+r,\mu)}.
\end{equation}
The proof is by induction on $r$. The base case $r = 0$ is clearly
true. For the induction step, suppose that we have for some $r > 0$
that
\begin{equation}
\label{pfeq2}
\sum_{s\geq r} \sum_{|\al|=s} 2^{\#\al} \,
\Om^{\cC(p,\la)}_{(p+s,\la-\al)} = \sum_{s\geq r}
\sum_{{\mu\rsa\la}\atop{|\mu|=|\la|-s}} 2^{m(\la/\mu)} \, \Om_{(p+s,\mu)}.
\end{equation}
Expanding the Giambelli formula with respect to the first component,
we obtain $\om_{p+s} \, \Om^{\cC(\la)}_{\la-\al}$ as the leading term of
$\Om^{\cC(p,\la)}_{(p+s,\la-\al)}$, while $\om_{p+s} \,\Om_\mu$ is the
leading term of $\Om_{(p+s,\mu)}$. Since the set of all products $\om_d\,
\Om_\nu$ for which $(d,\nu)$ is a $k$-strict partition is linearly
independent in $\Om^{(k)}$, we deduce from (\ref{pfeq2}) that
\begin{equation}
\label{pfeq3}
\sum_{|\al|=r} 2^{\#\al} \, \Om^{\cC(\la)}_{\la-\al} = 
\sum_{{\mu\rsa\la}\atop{|\mu|=|\la|-r}} 
2^{m(\la/\mu)} \, \Om_\mu.
\end{equation}
By applying Lemma \ref{basiclemD} to (\ref{pfeq3}), we see that
(\ref{pfeq1}) is true, and this completes the induction. This also
finishes the proof of Theorem \ref{mirrorW}, since the argument shows
that (\ref{pfeq3}) holds for every integer $r\geq 0$.
\end{proof}

\subsection{The set $\A$ and the integer $N(\A)$}
\label{Adef}
We say that boxes $[r,c]$ and $[r',c']$ are {\em $(k-1)$-related} if
$$|c-k|+r = |c'-k|+r'.$$ For example, the two grey boxes in 
the diagram of Figure \ref{kprrelated} are $(k-1)$-related. 
\begin{figure}
\centering
\pic{0.65}{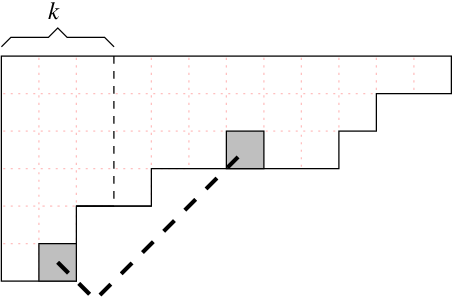} 
\caption{Two $(k-1)$-related boxes in a Young diagram}
\label{kprrelated}
\end{figure}
We call box $[r,c]$ a {\em left box} if $c \leq k$ and a {\em
right box} if $c>k$.

If $\mu\subset\la$ are two $k$-strict
partitions such that $\la/\mu$ is a $k'$-horizontal strip,
we define $\la_0=\mu_0=\infty$ and agree that the
diagrams of $\la$ and $\mu$ include all boxes $[0,c]$ in row zero.
We let $R$ (respectively $\A$) denote the set of right boxes of $\mu$
(including boxes in row zero) which are bottom boxes of $\la$ in their
column and are (respectively are not) $(k-1)$-related to a left box of
$\la/\mu$. Let $N(\A)$ denote the number of connected components of
$\A$. Moreover, define
\[
\wh{n}(\la/\mu)= \ell_k(\mu) - \ell_k(\la) + m(\la/\mu).
\]

\begin{lemma}
A pair $\mu\subset\la$ of $k$-strict partitions forms a $k'$-horizontal
strip $\la/\mu$ if and only if (i) $\la/\mu$ is contained in the rim
of $\la$, and the right boxes of $\la/\mu$ form a horizontal strip;
(ii) no two boxes in $R$ are $(k-1)$-related; and (iii) if two boxes of
$\la/\mu$ lie in the same column, then they are $(k-1)$-related to
exactly two boxes of $R$, which both lie in the same row. We have 
\begin{equation}
\label{whn}
\wh{n}(\la/\mu)=  
\begin{cases} 
N(\A) & \text{if $\la$ has
positive type and $\mu$ does not}, \\
N(\A)-1 & \text{otherwise}.
\end{cases}
\end{equation}
\end{lemma}
\begin{proof}
We have $\mu\rsa\la$ if and only if $\la\to(p+r,\mu)$ for any
$p\geq|\la|+2k$, where $r=|\la-\mu|$. Observe that a box of
$(p+r,\mu)\ssm\la$ corresponds to a box of $\mu$ which is a bottom box
of $\la$ in its column. The rest of the proof is a straightforward
translation of the definitions in \S \ref{pierirulegen}.
\end{proof}

\section{Reduction formulas and tableaux}
\label{red&tab}

\subsection{Schur and theta polynomials}
\label{theta1}
Let $x=(x_1,x_2,\ldots)$ and set $y=(y_1,\ldots,y_k)$ for a fixed
integer $k\geq 1$. Consider the generating series
\[
\prod_{i=1}^{\infty}\frac{1+x_it}{1-x_it} 
= \sum_{r=0}^{\infty}q_r(x)t^r
\ \ \ \text{and} \ \ \ 
\prod_{j=1}^k(1+y_jt) = \sum_{r=0}^{\infty}
e_r(y)t^r
\]
for the Schur $Q$-functions $q_r(x)$ and elementary symmetric
polynomials $e_r(y)$. If $\la$ is any partition, let $\la'$ denote the 
partition conjugate to $\la$, whose diagram is the transpose of the 
diagram of $\la$.  Then the Schur $S$-polynomial $s_{\la'}(y)$ may 
be defined by the equation
\[
s_{\la'} = \prod_{i<j}(1-R_{ij})\, e_\la.
\]
Furthermore, given any strict partition $\la$ of length
$\ell(\la)$, the Schur $Q$-function $Q_\la(x)$ is defined by the
raising operator expression
\[
Q_\la = \prod_{i<j}\frac{1-R_{ij}}{1+R_{ij}}\, q_\la
\]
and the $P$-function $P_\la(x)$ is given by $P_\la=
2^{-\ell(\la)}\,Q_\la$.  In particular $P_0=1$ and for each integer
$r\geq 1$, we have $P_r = q_r/2$.

Following \cite{BKT2}, for each integer $r$, define
$\ti_r=\ti_r(x\,;y)$ by
\[
\ti_r = \sum_{i\geq 0} q_{r-i}(x) e_i(y).
\]
We let $\Gamma^{(k)}= \Z[\ti_1,\ti_2,\ti_3,\ldots]$ be the ring of
theta polynomials.  There is a ring isomorphism
$\Om^{(k)}\to\Gamma^{(k)}$ sending $\om_r$ to $\ti_r$ for all $r$.  Let
$\la$ be a $k$-strict partition, and consider the raising operator
\[
\wt{R}^{\la} = \prod_{i<j}(1-R_{ij})\prod_{i<j\, :\, \la_i+\la_j > 2k+j-i}
(1+R_{ij})^{-1}.
\]
The theta polynomial $\Ti_\la(x\,;y)$ is defined by the equation
$\Ti_\la = \wt{R}^\la\, \ti_\la$. The polynomials $\Theta_\la$ for all
$k$-strict partitions $\la$ form a $\Z$-basis of $\Gamma^{(k)}$.

Set
\[
\eta_r(x\,;y) = \begin{cases}
e_r(y) + 2\sum_{i=0}^{r-1} P_{r-i}(x)e_i(y) & \text{if $r<k$}, \\
 \sum_{i=0}^r P_{r-i}(x) e_i(y) & \text{if $r\geq k$}
\end{cases}
\]
and $$\eta'_k(x\,;y) =  \sum_{i=0}^{k-1} P_{k-i}(x) e_i(y).$$ 
For any $r\geq 0$, we have
\[
\ti_r=
\begin{cases}
\eta_r &\text{if $r< k$},\\
\eta_k+\eta_k' &\text{if $r=k$},\\
2\eta_r &\text{if $r> k$},
\end{cases}
\]
while $\eta_k-\eta'_k=e_k(y)$. Following \cite{BKT3}, we define the
ring of eta polynomials
\[
B^{(k)} = \Z[\eta_1,\ldots, \eta_{k-1},\eta_k,\eta'_k,\eta_{k+1}\ldots].
\]

\subsection{The reduction formulas for $\wh{\Eta}_{\la}$ and $\wt{\Eta}_{\la}$}
For any $k$-strict partition $\la$, define $\wh{\Ti}_\la(x\,;y)$ and
$\wh{\Eta}_\la(x\,;y)$ by the equations
\[
\wh{\Ti}_\la = R^\la\, \ti_\la \ \ \mathrm{and} \ \ 
\wh{\Eta}_\la = 2^{-\ell_k(\la)}\,\wh{\Ti}_\la.
\]
and let $\tilde{x}=(x_2,x_3,\ldots)$.

\begin{prop}
\label{reductthm}
For any $k$-strict partition $\la$, we have the reduction formula 
\begin{equation}
\label{reductTH}
\wh{\Eta}_\la(x\,;y) = \sum_{p=0}^\infty \,x_1^p\,
\sum_{{\mu\rsa\la}\atop{|\mu|=|\la|-p}}
2^{\wh{n}(\la/\mu)} \, \wh{\Eta}_\mu(\tilde{x}\,;y).
\end{equation}
\end{prop}
\begin{proof}
We compute that
\[
\sum_{r=0}^{\infty}\ti_r(x\,;y)t^r = \frac{1+x_1t}{1-x_1t}\,
\prod_{i=2}^{\infty}\frac{1+x_it}{1-x_it} \,\prod_{j=1}^k
(1+y_jt) = \sum_{i=0}^{\infty}x_1^i\,2^{\#i}\,\sum_{s=0}^{\infty}
\ti_s(\tilde{x}\,;y)t^{s+i}
\]
and therefore, for any integer sequence $\mu$, we have
\begin{equation}
\label{stepto}
\ti_\mu(x\,;y) = \sum_{\al\geq 0}
x_1^{|\al|}\,2^{\#\al}\,\ti_{\mu-\al}(\tilde{x}\,;y)
\end{equation}
summed over all compositions $\al$. If $R$ denotes any raising operator,
then
\begin{equation*}
R\, \ti_\mu(x\,;y) = \ti_{R\mu}(x\,;y) =
\sum_{\al\geq 0} x_1^{|\al|}\,2^{\#\al}\,\ti_{R\mu-\al}(\tilde{x}\,;y) 
= \sum_{\al\geq 0} x_1^{|\al|}\,2^{\#\al}\,R\,\ti_{\mu-\al}(\tilde{x}\,;y).
\end{equation*}
Taking $\mu$ equal to a $k$-strict partition $\la$ and applying the
raising operator $R^\la$ to both sides of (\ref{stepto}), we therefore
obtain
\[
\wh{\Ti}_\la(x\,;y) = \sum_{\al\geq 0} x_1^{|\al|}\,2^{\#\al}\,
\wh{\Ti}^{\cC(\la)}_{\la-\al}(\tilde{x}\,;y) = 
\sum_{p=0}^\infty \,x_1^p \,\sum_{|\al|=p}2^{\#\al}\,
\wh{\Ti}^{\cC(\la)}_{\la-\al}(\tilde{x}\,;y),
\]
where $\wh{\Ti}^{\cC(\la)}_{\la-\al} = R^\la\,\ti_{\la-\al}$ by
definition. We now use the mirror identity (\ref{minuseq}) and 
the equation $\wh{\Ti}_\la = 2^{\ell_k(\la)}\,\wh{\Eta}_\la$
to complete the proof.
\end{proof}

For any $k$-strict partition $\la$ of positive type, we define
$\wt{\Eta}_\la(x\,;y)$ by the equation
\begin{equation}
\label{Hwt}
\wt{\Eta}_\la = 2^{-\ell_k(\la)}\,e_k(y)\,\Ti_{\la - k}
\end{equation}
where $\la-k$ means $\la$ with one part equal to $k$ removed. If $\la$
does not have a part equal to $k$, we agree that $\wt{\Eta}_\la = 0$.
The following analogue of Proposition \ref{reductthm} is valid for the
polynomials $\wt{\Eta}(x\,;y)$.

\begin{prop}
\label{reductthmwTH}
For any $k$-strict partition $\la$ of positive type, we have the 
reduction formula 
\begin{equation}
\label{reductwTH}
\wt{\Eta}_\la(x\,;y) = \sum_{p=0}^\infty \,x_1^p\,
\sum_{{\mu\rsa\la}\atop{|\mu|=|\la|-p}}
2^{\wh{n}(\la/\mu)} \, \wt{\Eta}_\mu(\tilde{x}\,;y)
\end{equation}
where the sum is over $k$-strict partitions $\mu\subset \la$ 
of positive type such that $\la/\mu$ is a $k'$-horizontal 
strip.
\end{prop}
\begin{proof}
This follows immediately from (\ref{Hwt}) and the reduction formula
for the theta polynomial $\Ti_{\la-k}(x\,;y)$. For the latter, see
\cite[Theorem 4]{T1}.
\end{proof}

\subsection{Typed $k$-strict partitions and eta polynomials}
A {\em typed $k$-strict partition} $\la$ consists of a $k$-strict
partition $(\la_1, \dots, \la_\ell)$ together with an integer
$\type(\la) \in \{0,1,2\}$, such that $\type(\la)>0$ if and only if
$\la_j=k$ for some index $j$. The type is usually omitted from the
notation for the pair $(\la,\type(\la))$. Suppose that $\la$ is a
typed $k$-strict partition and $R$ is any finite monomial in the
operators $R_{ij}$ which appears in the expansion of the power series
$R^\la$ in (\ref{giambelliCgen}).  If $\type(\la)=0$, then set $R
\star \ti_{\la} = \ti_{R \,\la}$. Suppose that $\type(\la)>0$, let $m$
be the smallest index such that $\la_m=k$, and set $\wh{\al} =
(\al_1,\ldots,\al_{m-1},\al_{m+1},\ldots,\al_\ell)$ for any integer
sequence $\al$ of length $\ell$. If $R$ involves any factors $R_{ij}$
with $i=m$ or $j=m$,\footnote{Note that if $i<m<j$, then the
  factorization $R_{ij}=R_{im}R_{mj}$ is not allowed.} then let $R
\star \ti_{\la} = \frac{1}{2}\,\ti_{R \,\la}$. If $R$ has no such
factors, then let
\[
R \star \ti_{\la} = \begin{cases}
\eta_k \,\ti_{\wh{R \,\la}} & \text{if  $\,\type(\la) = 1$}, \\
\eta'_k \, \ti_{\wh{R \,\la}} & \text{if  $\,\type(\la) = 2$}.
\end{cases}
\]
The {\em eta polynomial} $\Eta_\la=\Eta_\la(x\,;y)$ is the element of
$B^{(k)}$ defined by the raising operator formula

\begin{equation}
\label{Etadef}
\Eta_\la =2^{-\ell_k(\la)}R^{\la} \star \ti_{\la}.  
\end{equation}
We note that, following \cite{BKT3, T2, T3}, the term `eta polynomial'
is used to denote both the Giambelli polynomial (\ref{Etadef}) in the
formal variables $\eta_r$, $\eta'_k$ and its image in the ring
$B^{(k)}$ (we are only concerned with the latter here). The {\em
  type} of the polynomial $\Eta_\la$ is the same as the type of
$\la$. It was shown in \cite[Theorem 4]{BKT3} that the $\Eta_\la$ for
typed $k$-strict partitions $\la$ form a $\Z$-basis of the ring
$B^{(k)}$.


For any typed $k$-strict partition $\la$, we define $\wh{\Eta}_\la$
and $\wt{\Eta}_\la$ as above, by ignoring the type of $\la$. It follows that
the definition (\ref{Etadef}) of $\Eta_\la$ is equivalent to the formula
\begin{equation}
\label{Etaeq}
\Eta_\la = \begin{cases}
\wh{\Eta}_\la & \text{if $\type(\la)=0$}, \\
\frac{1}{2}(\wh{\Eta}_\la + \wt{\Eta}_\la) & \text{if $\type(\la)=1$}, \\
\frac{1}{2}(\wh{\Eta}_\la - \wt{\Eta}_\la) & \text{if $\type(\la)=2$}.
\end{cases}
\end{equation}

\begin{lemma}
\label{zerolem}
Let $\la$ be a typed $k$-strict partition and $\Eta_\la(0\,;y)$ be obtained 
from $\Eta_\la(x\,;y)$ by substituting $x_i=0$ for all $i\geq 1$. Then 
we have
\[
\Eta_\la(0\,;y) = 
\begin{cases} 
0 & \text{if $\la_1>k$}, \\
s_{\la'}(y) & \text{if $\la_1=k$ and $\type(\la)=1$}, \\
0 & \text{if $\la_1=k$ and $\type(\la)=2$}, \\
s_{\la'}(y) & \text{if $\la_1<k$}
\end{cases}
\]
where $\la'$ denotes the partition conjugate to $\la$. 
\end{lemma}
\begin{proof}
The raising operator definition of $\wh{\Ti}_\la$ gives
\begin{equation}
\label{Ti0la}
\wh{\Ti}_\la(0\,;y) = R^\la\, e_\la(y)
\end{equation}
where $e_\la = \prod_ie_{\la_i}(y)$ and $e_r(y)$ denotes the $r$-th
elementary symmetric polynomial in $y$. Since $e_r(y)=0$ for $r>k$, we
deduce from (\ref{Ti0la}) that $\wh{\Eta}_\la(0\,;y)=0$ unless
$\la_1\leq k$. In the latter case we have $\ell_k(\la)=0$,
$\cC(\la)=\emptyset$, and
\begin{equation}
\label{Eta0}
\wh{\Ti}_\la(0\,;y)=\Ti_\la(0\,;y) = \prod_{i<j}(1-R_{ij})\, e_\la(y) = 
s_{\la'}(y).
\end{equation}
Suppose that $\la_1=k$, so that $\type(\la)>0$. If $\type(\la)=1$, then 
equations (\ref{Hwt}), (\ref{Etaeq}) and (\ref{Eta0}) give 
\begin{align*}
\Eta_\la(0\,;y) & = \frac{1}{2}\left(\wh{\Ti}_\la(0\,;y) + 
e_k(y)\Ti_{\la-k}(0\,;y)\right) \\
& = \frac{1}{2}\left(s_{\la'}(y) + e_k(y)s_{(\la-k)'}(y)\right) 
= s_{\la'}(y),
\end{align*}
since the classical type A Pieri rule gives $e_k(y)s_{(\la-k)'}(y) =
s_{\la'}(y)$. We similarly compute that $\Eta_\la(0\,;y) = 0$ if
$\type(\la)=2$. Finally, if $\la_1<k$ then $\type(\la)=0$ and
$\Eta_\la(0\,;y) = \wh{\Ti}_\la(0\,;y)= s_{\la'}(y)$, using
(\ref{Eta0}) again.
\end{proof}

\subsection{Tableau formulas}
\label{tableaux}
In this subsection we will obtain a description of the eta
polynomial $\Eta_\la(x\,;y)$ as a sum over 
tableaux which are
fillings of the Young diagram of $\la$. We first prove a reduction
formula for the $x$ variables which appear in $\Eta_\la$.

\begin{defn}
If $\la$ and $\mu$ are typed $k$-strict partitions with
$\mu\subset\la$, we write $\mu\rsa\la$ and say that $\la/\mu$ is a
{\em typed $k'$-horizontal strip} if the underlying $k$-strict
partitions are such that $\la/\mu$ is a $k'$-horizontal strip and in
addition $\type(\la)+\type(\mu)\neq 3$. In this case we set
$n(\la/\mu) = N(\A)-1$, where the set $\A$ and integer $N(\A)$ 
are defined as in \S
\ref{Adef}.
\end{defn}

\begin{thm}
\label{reductthmH}
For any typed $k$-strict partition $\la$, we have the reduction formula 
\begin{equation}
\label{reductH}
\Eta_\la(x\,;y) = \sum_{p=0}^\infty \,x_1^p\,
\sum_{{\mu\rsa\la}\atop{|\mu|=|\la|-p}}
2^{n(\la/\mu)} \, \Eta_\mu(\tilde{x}\,;y)
\end{equation}
where $\tilde{x} = (x_2,x_3,\ldots)$ and the inner sum is over typed
$k$-strict partitions $\mu$ with $\mu\rsa\la$ and $|\mu|=|\la|-p$.
\end{thm}
\begin{proof}
Assume that $\type(\la)=1$; the proof in the other cases is 
similar. Using Definition \ref{Etadef} and equations (\ref{reductTH}) 
and (\ref{reductwTH}) we obtain
\begin{equation}
\label{etared}
\Eta_\la(x\,;y) = \sum_{p=0}^\infty \,x_1^p\,
\sum_{{\mu\rsa\la}\atop{|\mu|=|\la|-p}}
2^{\wh{n}(\la/\mu)-1} \, (\wh{\Eta}_\mu(\tilde{x}\,;y)
+ \wt{\Eta}_\mu(\tilde{x}\,;y))
\end{equation}
where the inner sum is over all $k$-strict partitions $\mu\subset\la$
with $|\mu|=|\la|-p$ such that $\la/\mu$ is a $k'$-horizontal
strip. The result follows by combining equation (\ref{etared}) with
(\ref{whn}) and (\ref{Etaeq}).
\end{proof}

Let {\bf P} denote the ordered alphabet
$\{\wh{1}<\wh{2}<\cdots<\wh{k}<1,1^\circ<2,2^\circ<\cdots\}$.  The symbols
$\wh{1},\ldots,\wh{k}$ are said to be {\em marked}, while the rest are {\em
unmarked}. Suppose that $\la$ is any typed $k$-strict
partition.

\begin{defn}
a) A {\em typed $k'$-tableau} $T$ of shape $\la/\mu$ is a sequence of typed
$k$-strict partitions
\begin{equation}
\label{Teq}
\mu = \la^0\subset\la^1\subset\cdots\subset\la^r=\la
\end{equation}
such that $\la^i/\la^{i-1}$ is a typed $k'$-horizontal strip for
$1\leq i\leq r$.  We represent $T$ by a filling of the boxes in
$\la/\mu$ with unmarked elements of {\bf P} which is weakly increasing
along each row and down each column, such that for each $i$, the boxes
in $T$ with entry $i$ or $i^\circ$ form the skew diagram
$\la^i/\la^{i-1}$, and we use $i$ (resp.\ $i^\circ$) if and only if
$\type(\la^i)\neq 2$ (resp.\ $\type(\la^i)=2$), for every $i\geq 1$.
A {\em standard typed $k'$-tableau} on $\la/\mu$ is a typed
$k'$-tableau $T$ of shape $\la/\mu$ such that the entries
$1,2,\ldots,|\la - \mu|$, circled or not, each appear exactly once in
$T$.  For any typed $k'$-tableau $T$ we define
\[
n(T)=\sum_i n(\la^i/\la^{i-1}) \quad \text{and} \quad
x^T = \prod_i x_i^{m_i}
\]
where $m_i$ denotes the number of times that $i$ or $i^\circ$
appears in $T$.

\medskip
\noin b) A {\em typed $k'$-bitableau} $U$ of shape $\la$ is a
filling of the boxes in $\la$ with elements of {\bf P} which is weakly
increasing along each row and down each column, such that (i) the
unmarked entries form a typed $k'$-tableau $T$ of shape $\la/\mu$ 
with $\type(\mu)\neq 2$, and (ii) the marked entries are a filling 
of $\mu$ which is strictly increasing along each row. We define
\[
n(U)=n(T) \quad \text{and} \quad
(xy)^U= x^T \,\prod_{j=1}^k y_j^{n_j} 
\]
where $n_j$ denotes the number of times that $\wh{j}$ appears in $U$.
\end{defn}

\begin{thm}
\label{Htableauxthm}
For any typed $k$-strict partition $\la$, we have 
\[
\Eta_\la(x\,;y) = \sum_U 2^{n(U)}(xy)^U 
\]
where the sum is over all typed $k'$-bitableaux $U$ of shape $\la$.
\end{thm}
\begin{proof}
Let $m$ be a positive integer, $x^{(m)} = (x_1,\dots,x_m)$, and
let  $\Eta_\la(x^{(m)}\,;y)$ be the result of substituting
$x_i=0$ for $i>m$ in $\Eta_\la(x\,;y)$. It follows from equation
(\ref{reductH}) that
\begin{equation}
\label{mreduct}
\Eta_\la(x^{(m)}\,;y) = \sum_{p=0}^\infty \,x_m^p\,
\sum_{{\mu\rsa\la}\atop{|\mu|=|\la|-p}}
2^{n(\la/\mu)} \, \Eta_\mu(x^{(m-1)}\,;y).
\end{equation}
Iterating equation (\ref{mreduct}) $m$ times produces
\[
\Eta_\la(x^{(m)}\,;y) = \sum_{\mu,T} \,2^{n(T)}\,x^T
\Eta_\mu(0\,;y)
\]
where the sum is over all typed $k$-strict partitions $\mu\subset\la$
and typed $k$-tableau $T$ of shape $\la/\mu$ with no entries greater
than $m$.  We deduce from Lemma \ref{zerolem} that $\Eta_\mu(0\,;y)
=0$ unless $\mu_1\leq k$ and $\type(\mu)\neq 2$, in which case
$\Eta_\mu(0\,;y) = s_{\mu'}(y)$. The combinatorial definition of Schur
$S$-functions \cite[I.(5.12)]{M} states that
\[
s_{\mu'}(y) = \sum_S y^S
\]
summed over all semistandard Young tableaux $S$ of shape $\mu'$ with
entries from $1$ to $k$. We conclude that
\[
\Eta_\la(x^{(m)}\,;y) = \sum_U \,2^{n(U)}\,(xy)^U
\]
summed over all typed $k'$-bitableaux $U$ of shape $\la$ with no
entries greater than $m$. The theorem follows by letting $m$ tend to
infinity.
\end{proof}

\begin{example} Let $k=1$, $\la=(3,1)$ with $\type(\la)=1$, 
and consider the alphabet $\text{\bf
  P}_{1,2}=\{\wh{1}<1,1^\circ<2,2^\circ\}$. There are thirteen typed
$1'$-bitableaux $U$ of shape $\la$ with entries in $\text{\bf
  P}_{1,2}$. The three typed $1'$-bitableaux $\dis \begin{array}{l}
  \wh{1}\,1\,2 \\ \wh{1}
\end{array}$, $\dis \begin{array}{l} \wh{1}\,2\,2 \\ 1
\end{array}$, and $\dis \begin{array}{l} \wh{1}\,1\,2 \\ 1
\end{array}$ satisfy $n(U)=1$, while the ten typed $1'$-bitableaux
\begin{gather*}
\begin{array}{l} 1\,1\,1 \\ 2 \end{array}  \ \ \
\begin{array}{l} 1\,1\,2 \\ 2 \end{array}  \ \ \
\begin{array}{l} 1\,2\,2 \\ 1 \end{array}  \ \ \
\begin{array}{l} 1\,2\,2 \\ 2 \end{array}  \ \ \
\begin{array}{l} \wh{1}\,1\,1 \\ 1 \end{array}  \\
\begin{array}{l} \wh{1}\,1\,1 \\ 2 \end{array}  \ \ \
\begin{array}{l} \wh{1}\,1\,2 \\ 2 \end{array}  \ \ \
\begin{array}{l} \wh{1}\,2\,2 \\ 2 \end{array}  \ \ \
\begin{array}{l} \wh{1}\,1\,1 \\ \wh{1} \end{array} \ \ \
\begin{array}{l} \wh{1}\,2\,2 \\ \wh{1} \end{array} 
\end{gather*}
satisfy $n(U)=0$. If $\Eta_{3,1}$ denotes the eta polynomial indexed
by $\la$, then we deduce from Theorem \ref{Htableauxthm} that
\begin{align*}
\Eta_{3,1}(x_1,x_2\,; y_1) &= (x_1^3x_2+2x_1^2x_2^2+x_1x_2^3) +
(x_1^3+3x_1^2x_2+3x_1x_2^2+x_2^3)\,y_1 \\ &\qquad +
(x_1^2+2x_1x_2+x_2^2)\,y_1^2 \\ &= P_{3,1}(x_1,x_2) +
\left(P_3(x_1,x_2)+P_{2,1}(x_1,x_2)\right)\,y_1 + P_2(x_1,x_2)\,
y_1^2.
\end{align*}
We similarly find that for $\la=(3,1)$ of type $2$, there are six
typed $1'$-bitableaux $U$ of shape $\la$ with entries in $\text{\bf
P}_{1,2}$, namely 
\begin{gather*}
\begin{array}{l} 1\,1\,1 \\ 2^\circ \end{array}  \ \ \
\begin{array}{l} 1\,1\,2^\circ \\ 2^\circ \end{array}  \ \ \
\begin{array}{l} 1^\circ\,2^\circ\,2^\circ \\ 1^\circ \end{array}  \ \ \
\begin{array}{l} 1^\circ\,2^\circ\,2^\circ \\ 2^\circ \end{array}  \ \ \
\begin{array}{l} \wh{1}\,1\,1 \\ 2^\circ \end{array}  \ \ \
\begin{array}{l} \wh{1}\,1\,2^\circ \\ 2^\circ \end{array} 
\end{gather*}
all of which satisfy $n(U)=0$.  If $\Eta'_{3,1}$ denotes the eta
polynomial indexed by $\la$, we deduce from Theorem \ref{Htableauxthm}
that
\begin{align*}
\Eta'_{3,1}(x_1,x_2\,; y_1) 
&= (x_1^3x_2+2x_1^2x_2^2+x_1x_2^3) 
+ (x_1^2x_2+x_1x_2^2)\,y_1 \\
&= P_{3,1}(x_1,x_2) + P_{2,1}(x_1,x_2)\,y_1.
\end{align*}
Notice that $\Eta_{3,1}+\Eta'_{3,1} = \frac{1}{2}\Ti_{3,1}$, and compare
with \cite[Example 7]{T1}.
\end{example}

Theorem \ref{Htableauxthm} motivates the following definition.

\begin{defn}
\label{Edefn}
For $\la$ and $\mu$ any two typed $k$-strict partitions with 
$\mu\subset\la$, let
\[
E_{\la/\mu}(x)=\sum_T 2^{n(T)}\,x^T
\]
where the sum is over all typed $k'$-tableaux $T$ of shape $\la/\mu$.
\end{defn}

\begin{example}
\label{4ex}
(a) There is an involution $j$ on the set of all typed $k$-strict
partitions, which is the identity on partitions of type $0$ and
exchanges type $1$ and type $2$ partitions of the same
shape. If $T$ is a typed $k'$-tableau of shape $\la/\mu$ then we let
$j(T)$ be the typed $k'$-tableau of shape $j(\la)/j(\mu)$ obtained by
applying $j$ to each typed partition $\la^i$ which appears in the
sequence (\ref{Teq}) determined by $T$. Then $j$ is an involution on
the set of all typed $k'$-tableaux and it follows from Definition
\ref{Edefn} that $$E_{\la/\mu}(x) = E_{j(\la)/j(\mu)}(x).$$

\medskip
\noin (b) Suppose that $\mu_i \geq \min(k,\la_i)$ for all $i$. Then
Definition \ref{Edefn} becomes the tableau based definition of skew
Schur $P$-functions found e.g.\ in \cite[III.(8.16)]{M}.  We deduce
that $$E_{\la/\mu}(x)=P_{\la/\mu}(x).$$

\medskip
\noin 
(c) Suppose that $\la_1<k$, so in particular
$\cC(\la)=\emptyset$. Then $n(\la/\mu)$ is equal to the number of
edge-connected components of $\la/\mu$, for any partition
$\mu\subset\la$. It follows from Worley \cite[\S 2.7]{W} that
\[
\sum_T 2^{n(T)}\, x^T = S_{\la/\mu}(x):=
\det(q_{\la_i-\mu_j+j-i}(x))_{i,j}.
\]
We therefore have $E_{\la/\mu}(x) =  S_{\la/\mu}(x)$ and
\[
\Eta_\la(x\,;y) = \sum_{\mu\subset\la}S_{\la/\mu}(x)\,s_{\mu'}(y),
\]
in agreement with \cite[Prop.\ 5.4(a)]{BKT3}.
\end{example}

\medskip
\noin 
(d) Suppose that there is only one variable $x$. Then we have
\[
E_{\la/\mu}(x) = \begin{cases}
2^{n(\la/\mu)}\, x^{|\la-\mu|} & \text{if $\la/\mu$ is a typed
$k'$-horizontal strip}, \\
0 & \text{otherwise}.
\end{cases}
\]

\begin{cor}
{\em (a)} Let $x=(x_1,x_2,\ldots)$ and $x'=(x'_1,x'_2,\ldots)$ be two
sets of variables, and let $\la$ be any typed $k$-strict
partition. Then we have
\begin{equation}
\label{mastereqH}
\Eta_\la(x,x'\,;y) = \sum_{\mu\subset\la} E_{\la/\mu}(x)\,
\Eta_\mu(x'\,;y)\, ,
\end{equation}
\begin{equation}
\label{master2H}
\Eta_\la(x\,;y) = \sum_{\mu\subset\la} E_{\la/\mu}(x)\,
s_{\mu'}(y)\, ,
\end{equation}
and 
\begin{equation}
\label{master3H}
E_\la(x,x') = \sum_{\mu\subset\la} E_{\la/\mu}(x)\,
E_{\mu}(x'),
\end{equation}
where the sums are over all typed $k$-strict partitions $\mu\subset\la$.

\medskip
\noin
{\em (b)} For any two typed $k$-strict partitions $\la$ and $\mu$ with
$\mu\subset\la$, we have
\begin{equation}
\label{EEE}
E_{\la/\mu}(x,x') = \sum_{\nu} E_{\la/\nu}(x)\,
E_{\nu/\mu}(x'),
\end{equation}
where the sum is over all typed $k$-strict partitions $\nu$
with $\mu\subset\nu\subset\la$.
\end{cor}
\begin{proof}
The proof is the same as that for Corollaries 6 and 7 in \cite{T1}.
\end{proof}

Since $\Eta_\la(x,x'\,;y)$ is symmetric in the variables $(x,x')$ and
the $\Eta_\mu(x'\,;y)$ are linearly independent, it follows
immediately from identity (\ref{mastereqH}) that $E_{\la/\mu}(x)$ is a
{\em symmetric function} in the variables $x$.  These functions will
be studied further in the next section.

\section{Stanley symmetric functions and skew elements}
\label{sss}

\subsection{Type D Stanley functions and Schubert polynomials}
\label{ssfns}
Let $\wt{W}_{n+1}$ be the Weyl group for the root system
of type $\text{D}_{n+1}$, and set $\wt{W}_\infty = \bigcup_n \wt{W}_{n+1}$.
The elements of $\wt{W}_{n+1}$ may be
represented as signed permutations of the set $\{1,\ldots,n+1\}$; we
will denote a sign change by a bar over the corresponding entry. The
group $\wt{W}_{n+1}$ is generated by the simple transpositions
$s_i=(i,i+1)$ for $1\leq i\leq n$, and an element $s_0$ which acts on
the right by
\[
(u_1,u_2,\ldots,u_{n+1})s_0 = (\ov{u}_2, \ov{u}_1, u_3 , \ldots, u_{n+1}).
\]
Every element $w\in \wt{W}_\infty$ can be expressed as a product of
$w=s_{a_1}s_{a_2}\cdots s_{a_r}$ of simple reflections $s_i$. If the
length $r$ of such an expression is minimal then the sequence of
indices $(a_1,\ldots, a_r)$ is called a {\em reduced word} for $w$,
and $r$ is called the {\em length} of $w$, denoted $\ell(w)$.  A
factorization $w=uv$ in $\wt{W}_\infty$ is {\em reduced} if
$\ell(w)=\ell(u)+\ell(v)$.

Following \cite{FS, FK1, FK2, L}, we will use the nilCoxeter algebra
$\cW_{n+1}$ of $\wt{W}_{n+1}$ to define type D Stanley symmetric
functions and Schubert polynomials.  $\cW_{n+1}$ is the free
associative algebra with unity generated by the elements
$u_0,u_1,\ldots,u_n$ modulo the relations
\[
\begin{array}{rclr}
u_i^2 & = & 0 & i\geq 0\ ; \\
u_0 u_1 & = & u_1 u_0 \\
u_0 u_2 u_0 & = & u_2 u_0 u_2 \\
u_iu_{i+1}u_i & = & u_{i+1}u_iu_{i+1} & i>0\ ; \\
u_iu_j & = & u_ju_i & j> i+1, \ \text{and} \ (i,j) \neq (0,2).
\end{array}
\]
For any $w\in \wt{W}_{n+1}$, choose a reduced word $a_1\cdots a_\ell$
for $w$ and define $u_w = u_{a_1}\ldots u_{a_\ell}$.  Since the last
four relations listed are the Coxeter relations for $D_{n+1}$, it is
clear that $u_w$ is well defined, and that the $u_w$ for $w\in
\wt{W}_{n+1}$ form a free $\Z$-basis of $\cW_{n+1}$.

Let $t$ be an indeterminate and define
\begin{gather*}
A_i(t) = (1+t u_n)(1+t u_{n-1})\cdots 
(1+t u_i) \ ; \\
D(t) = (1+t u_n)\cdots (1+t u_2)(1+t u_1)(1+t u_0)
(1+t u_2)\cdots (1+t u_n).
\end{gather*}
According to \cite[Lemma 4.24]{L}, for any commuting variables $s$,
$t$ we have $D(s)D(t) = D(t)D(s)$. Set $x=(x_1,x_2,\ldots)$ and
consider the product $D(x):=D(x_1)D(x_2)\cdots$.  We deduce
that the functions $E_w(x)$ in the formal power series expansion
\begin{equation}
\label{defE}
D(x) = \sum_{w\in\wt{W}_{n+1}}E_w(x)\, u_w
\end{equation}
are symmetric functions in $x$. The $E_w$ are the {\em type D Stanley 
symmetric functions}, introduced and studied in \cite{BH, L}.

Let $y=(y_1,y_2,\ldots)$. The Billey-Haiman {\em type D Schubert
polynomials} $\DS_w(x\,;y)$ for $w\in \wt{W}_{n+1}$ are defined by
expanding the formal product
\begin{equation}
\label{defD}
D(x) A_1(y_1)A_2(y_2)\cdots A_n(y_n)
= \sum_{w\in \wt{W}_{n+1}}\DS_w(x\,;y)\, u_w.
\end{equation}
The above definition is equivalent to the one in \cite{BH}.  Observe
that $\DS_w(x\,;y)$ is a polynomial in the $y$ variables, whose
coefficients are formal power series in the $x$ variables.  One checks
that $\DS_w$ is stable under the natural inclusion of $\wt{W}_n$ in
$\wt{W}_{n+1}$, and hence well defined for $w\in \wt{W}_\infty$. We
also deduce the following result from (\ref{defE}) and (\ref{defD}).

\begin{prop}
Let $w\in \wt{W}_\infty$ and $x'=(x'_1,x'_2,\ldots)$. Then we have
\begin{equation}
\label{Exx}
E_w(x,x') = \sum_{uv=w} E_u(x)E_v(x')
\end{equation}
and 
\begin{equation}
\label{besteq}
\DS_w(x,x'\,;y) = \sum_{uv=w}E_u(x)\,\DS_v(x'\,;y)
\end{equation} 
where the sums are over all reduced factorizations $uv=w$ in $\wt{W}_\infty$.
\end{prop}

\subsection{The Grassmannian elements of $\wt{W}_\infty$}
\label{wsec}
For $k\neq 1$, an element $w\in \wt{W}_\infty$ is $k$-Grassmannian if
$\ell(ws_i)=\ell(w)+1$ for all $i\neq k$. We say that $w$ is
$1$-Grassmannian if $\ell(ws_i)=\ell(w)+1$ for all $i\geq 2$.  The
elements of $\wt{W}_{n+1}$ index the Schubert classes in the
cohomology ring of the flag variety $\SO_{2n+2}/B$, which contains
$\HH^*(\OG(n+1-k,2n+2),\Z)$ as the subring spanned by Schubert classes
given by $k$-Grassmannian elements.  If $\wt{\cP}(k,n)$ denotes the
set of typed $k$-strict partitions whose diagrams fit inside a
rectangle of size $(n+1-k)\times (n+k)$, then each element $\la$ in
$\wt{\cP}(k,n)$ corresponds to a $k$-Grassmannian element $w_\la \in
\wt{W}_{n+1}$ which we proceed to describe.

The typed $k$-strict partitions $\la$ in $\wt{\cP}(k,n)$ are those
whose Young diagram fits inside the non-convex polygon $\Pi$ obtained
by attaching an $(n+1-k)\times k$ rectangle to the left side of a
staircase partition with $n$ rows.  When $n=7$ and $k=3$, the polygon
$\Pi$ looks as follows.
\[ \Pi \ \ = \ \ \ \raisebox{-36pt}{\pic{.6}{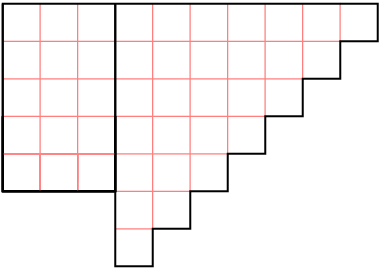}} \] 
The boxes of the staircase partition that are outside $\la$ lie in
south-west to north-east diagonals.  Such a diagonal is called {\em
related} if it is $k'$-related to one of the bottom boxes in the first
$k$ columns of $\la$, or to any box $[0,i]$ for which $\la_1 < i \leq
k$; the remaining diagonals are called {\em non-related}. Let $u_1 <
u_2 < \dots < u_k$ denote the lengths of the related diagonals.
Moreover, let $\la^1$ be the strict partition obtained by removing the
first $k$ columns of $\la$, let $\la^2$ be the partition of boxes
contained in the first $k$ columns of $\la$, and set
$r=\ell(\la^1)=\ell_k(\la)$.
\[ \pic{0.50}{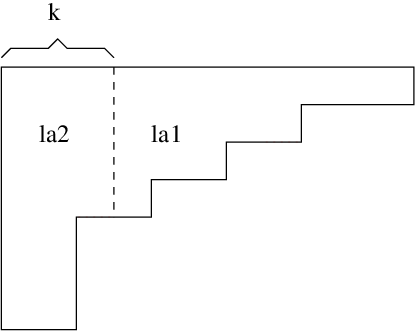} \]
If $\type(\la)=1$, then the $k$-Grassmannian element corresponding to
$\la$ is given by
\[
w_\la =
(u_1+1,\dots,u_k+1, \ov{(\la^1)_1+1}, \dots, \ov{(\la^1)_r+1},\wt{1},
v_1+1,\dots,v_{n-k-r}+1),
\] 
while if $\type(\la)=2$, then 
\[
w_\la =
(\ov{u_1+1},\dots,u_k+1, \ov{(\la^1)_1+1}, \dots, \ov{(\la^1)_r+1},\wt{1},
v_1+1,\dots,v_{n-k-r}+1),
\] 
where $v_1 < \dots < v_{n-k-r}$ are the lengths of the non-related
diagonals. Here $\wt{1}$ is equal to $1$ or $\ov{1}$, so that the
total number of barred entries in $w_\la$ is even. Finally, if
$\type(\la)=0$, then $u_1=0$, that is, one of the related diagonals
has length zero. In this case
\[
w_\la =
(\wt{1},u_2+1, \dots,u_k+1, \ov{(\la^1)_1+1}, \dots, \ov{(\la^1)_r+1},
v_1+1,\dots,v_{n+1-k-r}+1),
\] 
where $v_1 < \dots < v_{n+1-k-r}$ are the lengths of the non-related 
diagonals. For example, the element $\la
= (8,4,3,2) \in \wt{\cP}(3,7)$ of type $1$ 
corresponds to $w_\la = (3,5,7,\ov{6},\ov{2},1,4,8)$.
\[
\lambda \ = \ \ \raisebox{-53pt}{\pic{.6}{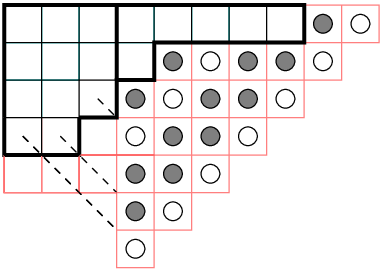}}
\]

The element $w_\la\in \wt{W}_\infty$ depends on $\la$ and $k$, but is
independent of $n$. Furthermore, it was proved in \cite[Prop.\
6.3]{BKT3} that
\begin{equation}
\label{TSeq}
\Eta_\la(x\,;y) = \DS_{w_\la}(x\,;y)
\end{equation}
for any typed $k$-strict partition $\la$. The main point of
(\ref{TSeq}) is the expression of $\DS_{w_\la}$ via raising operators,
as in (\ref{Etadef}). Observe also that the equality (\ref{TSeq}) is
taking place in the full ring $B^{(0)}[z]=\Z[P_1(x),P_2(x),\ldots ;
  z_1,z_2,\ldots]$ of type D Billey-Haiman Schubert polynomials, where
there are relations among the generators $P_r(x)$. Moreover, these
relations are used crucially in its proof, which is given in
\cite{BKT3}.

\subsection{The skew elements of $\wt{W}_\infty$} 
The following definition can be formulated for any Coxeter group.

\begin{defn}
An element $w\in \wt{W}_\infty$ is called {\em skew} if 
there exists a typed $k$-strict partition $\la$ (for some $k$)
and a reduced factorization $w_\la = ww'$ in $\wt{W}_\infty$.
\end{defn}

Note that if we have a reduced factorization $w_\la = ww'$ in
$\wt{W}_{n+1}$ for some typed $k$-strict partition $\la\in
\wt{\cP}(k,n)$, then the right factor $w'$ is $k$-Grassmannian, and
therefore equal to $w_\mu$ for some typed $k$-strict partition $\mu\in
\wt{\cP}(k,n)$.

\begin{prop}
\label{skewprop}
Suppose that $w$ is a skew element of $\wt{W}_\infty$, and let 
$\la$ and $\mu$ be typed 
$k$-strict partitions such that the factorization $w_\la = ww_\mu$
is reduced. Then we have
$\mu\subset\la$ and $E_w(x) = E_{\la/\mu}(x)$. 
\end{prop}
\begin{proof}
By combining (\ref{TSeq}) with (\ref{mastereqH}) and (\ref{besteq}),
we see that 
\begin{equation}
\label{3eq}
\sum_{\mu\subset\la} E_{\la/\mu}(x)\, \Eta_\mu(x'\,;y) = 
\Eta_\la(x,x'\,; y) = \sum_{uv=w_\la}E_u(x)\,\DS_v(x'\,;y)
\end{equation}
where the first sum is over all typed $k$-strict partitions
$\mu\subset \la$ and the second sum is over all reduced factorizations
$uv=w_\la$. The right factor $v$ in any such reduced factorization 
must equal $w_\nu$ for some typed $k$-strict partition $\nu$,
and therefore $\DS_v(x'\,;y) = \Eta_\nu(x'\,;y)$.  Since the
$\Eta_\nu(x'\,;y)$ for $\nu$ a typed $k$-strict partition form a
$\Z$-basis for the ring $B^{(k)}(x'\,; y)$ of eta polynomials in $x'$
and $y$, the desired result follows.
\end{proof}

\begin{defn}
Let $\la$ and $\mu$ be typed $k$-strict partitions in $\wt{\cP}(k,n)$ with
$\mu\subset\la$.  We say that $(\la,\mu)$ is a {\em compatible pair}
if there is a reduced word for $w_\la$ whose last $|\mu|$ entries form
a reduced word for $w_\mu$; equivalently, if we have $\ell(w_\la
w_\mu^{-1}) = |\la - \mu|$. 
\end{defn}

From Proposition \ref{skewprop} we immediately deduce the next result.

\begin{cor}
\label{compcor}
Let $(\la,\mu)$ be a compatible pair of typed $k$-strict partitions.
Then there is a 1-1 correspondence between reduced factorizations
of $w_\la w_\mu^{-1}$ and typed $k$-strict partitions $\nu$ with 
$\mu\subset\nu\subset\la$ such that $(\la,\nu)$ and $(\nu,\mu)$ are
compatible pairs.
\end{cor}

\begin{thm}
\label{abprop}
Let $\la$ and $\mu$ be typed $k$-strict partitions in $\wt{\cP}(k,n)$ with
$\mu\subset\la$.  Then the following conditions are equivalent: {\em
(a)} $E_{\la/\mu}(x)\neq 0$; {\em (b)} $(\la,\mu)$ is a
compatible pair; {\em (c)} there exists a standard typed $k'$-tableau on
$\la/\mu$. If any of these conditions holds, then
$E_{\la/\mu}(x) = E_{w_\la w_\mu^{-1}}(x)$.
\end{thm}
\begin{proof}
Equation (\ref{3eq}) may be rewritten in the form
\begin{equation}
\label{2eq}
\sum_{\mu\subset\la} E_{\la/\mu}(x)\, \Eta_\mu(x'\,;y) =
\sum_{\mu} E_{w_\la w_\mu^{-1}}(x)\,\Eta_\mu(x'\,;y)
\end{equation}
where the second sum is over all $\mu\subset\la$ such that $(\la,\mu)$
is a compatible pair. It follows that
\[
E_{\la/\mu}(x) =
\begin{cases}
E_{w_\la w_\mu^{-1}}(x) & \text{if $(\la,\mu)$ is a compatible pair}, \\
0 & \text{otherwise}.
\end{cases}
\]
Since clearly $E_w(x)\neq 0$ for any $w\in \wt{W}_\infty$, we deduce
that (a) and (b) are equivalent. Suppose now that $(\la,\mu)$ is a
compatible pair with $|\la|=|\mu|+1$, so that $w_\la = s_iw_\mu$ for
some $i\geq 0$. Observe that if $x$ is a single variable, then
$E_{s_i}(x) = x$, if $i\leq 1$, and $E_{s_i}(x)=2x$, if
$i>1$. Therefore $E_{\la/\mu}(x)\neq 0$, and we deduce from Example
\ref{4ex}(d) that $\la/\mu$ must be a typed $k'$-horizontal strip.
Using Corollary \ref{compcor}, it follows that there is a 1-1
correspondence between reduced words for $w_\la w_\mu^{-1}$ and
sequences of typed $k$-strict partitions
\[
\mu = \la^0\subset\la^1\subset\cdots\subset\la^r=\la
\]
such that $|\la^i|=|\la^{i-1}|+1$ and $\la^i/\la^{i-1}$ is a typed
$k'$-horizontal strip for $1\leq i\leq r=|\la-\mu|$. The latter objects 
are exactly the standard typed $k'$-tableaux on $\la/\mu$. This shows that 
(b) implies (c), and the converse is also clear.
\end{proof}

The previous results show that the non-zero terms in equations
(\ref{mastereqH}) and (\ref{master3H}) correspond exactly to the terms
in equations (\ref{besteq}) and (\ref{Exx}), respectively, when 
$w=w_\la$ is a $k$-Grassmannian element of $\wt{W}_\infty$.

\begin{cor}
\label{stdcor}
Let $w\in \wt{W}_\infty$ be a skew element and $(\la,\mu)$ be a
compatible pair such that $w_\la = w w_\mu$.  Then the number of
reduced words for $w$ is equal to the number of standard typed
$k'$-tableaux on $\la/\mu$.
\end{cor}

Let $\iota$ denote the involution on the set of reduced words which
interchanges the letters $0$ and $1$, for example $\iota(02120) =
12021$. Then the restriction of $\iota$ to the set of reduced words
representing skew elements of $\wt{W}_{\infty}$ corresponds to the
involution $j$ on the set of typed $k'$-tableaux defined in Example
\ref{4ex}(a).

\begin{example}
Let $\la$ be a typed $k$-strict partition with $\type(\la)\in \{0,1\}$
and let $\la^1$ and $\la^2$ be defined as in \S \ref{wsec}. We can form a
standard typed $k'$-tableau on $\la$ by filling the boxes of
$\la^2$, going down the columns from left to right, and then
filling the boxes of $\la^1$, going across the rows from top to
bottom.  When $k=2$ and $\la = (7,6,5,2)$ with $\type(\la)=1$, the
typed $2'$-tableau on $\la$ which results is
\[
\begin{array}{ccccccc}
1 & 5  & 9  & 10 & 11 & 12 & 13 \\
2 & 6  & 14 & 15 & 16 & 17 & \\ 
3 & 7  & 18 & 19 & 20 & & \\
4 & 8  & & & & & 
\end{array}
\]
which corresponds to the reduced word
\[
3\,2\,0\,4\,3\,2\,1\,5\,4\,3\,2\,0\,
4\,3\,2\,1\,5\,4\,3\,2
\]
for the $2$-Grassmannian element $w_\la=23\ov{6}\ov{5}\ov{4}\ov{1}\in
\wt{W}_6$.  We can form a reduced word for the element
$w'_\la=\ov{2}3\ov{6}\ov{5}\ov{4}1$ associated to $\la = (7,6,5,2)$
with $\type(\la)=2$ by applying the involution $\iota$ to obtain
\[
3\,2\,1\,4\,3\,2\,0\,5\,4\,3\,2\,1\,
4\,3\,2\,0\,5\,4\,3\,2.
\]
This last word corresponds to the standard typed $2'$-tableau
\[
\begin{array}{ccccccc}
1 & 5^\circ  & 9^\circ  & 10^\circ & 11^\circ & 12^\circ & 13^\circ \\
2 & 6^\circ  & 14^\circ & 15^\circ & 16^\circ & 17^\circ & \\ 
3 & 7^\circ  & 18^\circ & 19^\circ & 20^\circ & & \\
4 & 8^\circ  & & & & & 
\end{array}.
\]
\end{example}

\begin{example}
Consider the $1$-Grassmannian element $w=3\ov{4}\ov{2}1$ in $\wt{W}_4$ 
associated to the typed $1$-strict partition $\la=(4,2,1)$ of type 
$1$. The following table lists the nine reduced words for $w$ and the 
corresponding standard typed $1'$-tableaux of shape $\la$.
\smallskip

\[
\begin{array}{|cc|cc|cc|} \hline
\text{word} & \text{tableau} & \text{word} & \text{tableau} &
\text{word} & \text{tableau} \\ \hline 
1320321 & 
\begin{array}{cccc}  &&& \\ 
1 & 4 & 5 & 6 \\ 2 & 7 && \\ 3 &&& \\ &&& \end{array} & 
1323021 & 
\begin{array}{cccc} &&& \\ 
1 & 3 & 5 & 6 \\ 2 & 7 && \\ 4 &&& \\ &&& \end{array} &
1232021 & 
\begin{array}{cccc} &&& \\ 
1 & 3 & 4 & 5 \\ 2 & 7 && \\ 6 &&& \\ &&& \end{array} \\ \hline
1230201 & 
\begin{array}{cccc} &&& \\
1 & 2 & 3 & 5 \\ 4 & 7 && \\ 6 &&& \\ &&& \end{array} & 
1230210 & 
\begin{array}{cccc} &&& \\
1^\circ & 2 & 3 & 5 \\ 4 & 7 && \\ 6 &&& \\ &&& \end{array} &
1203201 & 
\begin{array}{cccc} &&& \\
1 & 2 & 3 & 4 \\ 5 & 7 && \\ 6 &&& \\ &&& \end{array} \\ \hline
1203210 & 
\begin{array}{cccc} &&& \\ 
1^\circ & 2 & 3 & 4 \\ 5 & 7 && \\ 6 &&& \\ &&& \end{array} & 
3120321 & 
\begin{array}{cccc} &&& \\
1 & 4 & 5 & 7 \\ 2 & 6 && \\ 3 &&& \\ &&& \end{array} &
3123021 & 
\begin{array}{cccc} &&& \\
1 & 3 & 5 & 7 \\ 2 & 6 && \\ 4 &&& \\ &&& \end{array} \\ \hline
\end{array}  
\]
\end{example}

\smallskip

\begin{cor}
\label{bhcor}
For any two typed $k$-strict partitions $\la$, $\mu$ with $\mu\subset\la$, 
the function $E_{\la/\mu}(x)$ is a nonnegative integer linear
combination of Schur $P$-functions.
\end{cor}
\begin{proof}
According to \cite{BH} and \cite{L}, the type D Stanley symmetric
function $E_w(x)$ is a nonnegative integer linear combination of Schur
$P$-functions. The result follows from this fact together with Theorem
\ref{abprop}.
\end{proof}

As noted in the introduction, the skew elements of the symmetric group
may be identified with the $321$-avoiding permutations, following
\cite{BJS, Ste1}.  It would be interesting to determine whether the
skew elements in $\wt{W}_\infty$ (and in the hyperoctahedral group)
can also be characterized by pattern avoidance conditions.


\begin{thebibliography}{BKT1}


\bibitem[BH]{BH} S. Billey and M. Haiman :
{\em Schubert polynomials for the classical groups},
J. Amer. Math. Soc. {\bf 8} (1995), 443--482.


\bibitem[BJS]{BJS} S. Billey, W. Jockusch and R. P. Stanley :
{\em Some combinatorial properties of Schubert polynomials},
J. Algebraic Combin. {\bf 2} (1993), no. 4, 345--374.



\bibitem[BKT1]{BKT1} A. S. Buch, A. Kresch and H. Tamvakis :
{\em Quantum Pieri rules for isotropic Grassmannians},
Invent. Math. {\bf 178} (2009), 345--405.


\bibitem[BKT2]{BKT2} A. S. Buch, A. Kresch and H. Tamvakis :
{\em A Giambelli formula for isotropic Grassmannians},
arXiv:0811.2781.


\bibitem[BKT3]{BKT3} A. S. Buch, A. Kresch and H. Tamvakis :
{\em A Giambelli formula for even orthogonal Grassmannians},
J. reine angew. Math., posted on September 3, 2013, 
DOI: 10.1515/crelle-2013-0071 (to appear in print).

\bibitem[FK1]{FK1} S. Fomin and A. N. Kirillov : 
{\em The Yang-Baxter equation, symmetric functions, and Schubert 
polynomials}, Discrete Math. {\bf 153} (1996), 123--143.


\bibitem[FK2]{FK2} S. Fomin and A. N. Kirillov :
{\em Combinatorial $B_n$-analogs of Schubert polynomials},
Trans. Amer. Math. Soc. {\bf 348} (1996), 3591--3620.


\bibitem[FS]{FS} S. Fomin and R. P. Stanley :
{\em Schubert polynomials and the nil-Coxeter algebra}, 
Adv. in Math. {\bf 103} (1994), 196--207.

\bibitem[Fu]{Fu} W. Fulton :
{\em Determinantal formulas for orthogonal and symplectic degeneracy
loci}, J. Differential Geom. {\bf 43} (1996), 276--290.


\bibitem[L]{L} T. K. Lam : {\em B and D analogues of stable Schubert
polynomials and related insertion algorithms}, Ph.D.\ thesis, M.I.T., 1994;
available at http://hdl.handle.net/1721.1/36537.

\bibitem[Li]{Li} D. E. Littlewood :
{\em On certain symmetric functions},
Proc. London Math. Soc. (3) {\bf 11} (1961), 485--498.



\bibitem[M]{M} I. Macdonald :
{\em Symmetric Functions and Hall Polynomials}, Second edition,
Clarendon Press, Oxford, 1995.


\bibitem[S1]{Ste1} J. R. Stembridge :
{\em On the fully commutative elements of Coxeter groups},
J. Algebraic Combin. {\bf 5} (1996), 353--385. 


\bibitem[S2]{Ste2} J. R. Stembridge :
{\em Some combinatorial aspects of reduced words in finite Coxeter 
groups}, Trans. Amer. Math. Soc. {\bf 349} (1997), 1285--1332.


\bibitem[T1]{T1} H. Tamvakis : 
{\em Giambelli, Pieri, and tableau formulas via raising operators}, 
J. reine angew. Math. {\bf 652} (2011), 207--244.

\bibitem[T2]{T2} H.  Tamvakis: {\em A Giambelli formula for classical
$G/P$ spaces}, J. Algebraic Geom., posted on July 11, 2013, 
PII S 1056-3911(2013)00604-9 (to appear in print).

\bibitem[T3]{T3} H.  Tamvakis: {\em Giambelli and degeneracy locus formulas 
for classical $G/P$ spaces}, arXiv:1305.3543.


\bibitem[W]{W} D. R. Worley : 
{\em A theory of shifted Young tableaux}, Ph.D. thesis, MIT, 1984.


\bibitem[Y]{Y} A. Young : 
{\em On quantitative substitutional analysis VI}, Proc. Lond. 
Math. Soc. (2) {\bf 34} (1932), 196--230.



\end{thebibliography}
\end{document}